\documentclass[12pt,leqno,amsfonts,amscd]{amsart}
\usepackage{epsfig}
\usepackage{amsmath}
\usepackage{amsfonts}
\usepackage{amssymb}
\usepackage{color}
\usepackage{hyperref}

\newtheorem{theorem}{Theorem}[section]
\newtheorem{lemma}[theorem]{Lemma}
\newtheorem{prop}[theorem]{Proposition}
\newtheorem{cor}[theorem]{Corollary}

\theoremstyle{definition}
\newtheorem{definition}[theorem]{Definition}

\theoremstyle{remark}

\newtheorem{remark}[theorem]{Remark}

\numberwithin{equation}{section}

\newcommand{\RR}{{\mathbb R}}

\newcommand{\eps}{\varepsilon}
\newcommand{\out}[1]{\ }

\DeclareMathOperator{\mflim}{mf{-}lim}

\DeclareMathOperator{\mfliminf}{mf{-}lim\,inf}

\let\cal=\mathcal
\renewcommand{\phi}{\varphi}

\begin{document}
\title[Sweeping at the Martin boundary]{Sweeping
at the Martin boundary of a fine domain}

\author{Mohamed El Kadiri}
\address{Universit\'e Mohammed V
\\D\'epartement de Math\'ematiques
\\Facult\'e des Sciences
\\B.P. 1014, Rabat
\\Morocco}
\email{elkadiri@fsr.ac.ma}

\author{Bent Fuglede}
\address{Department of Mathematical Sciences
\\Universitetsparken 5
\\2100 Copenhagen
\\Danmark}
 \email{fuglede@math.ku.dk}

\begin{abstract}
We study sweeping on a subset of the Riesz-Martin space of a fine domain in
$\RR^n$ ($n\ge2$), both with respect to the natural topology and the
minimal-fine topology, and show that the two notions of sweeping are
identical.
\end{abstract}

\maketitle

\section{Introduction}\label{sec1}

The fine topology on an open set $\Omega\subset\RR^n$ was introduced
by H.\ Cartan in classical potential theory. It is defined as the
smallest topology on $\Omega$ in which every super\-harmonic function on
$\Omega$ is continuous.
Potential theory on a finely open set, for example in $\RR^n$, was introduced
and studied in the 1970's by the second named author
\cite{F1}. The harmonic and super\-harmonic functions and the potentials in this
theory are termed finely [super]harmonic functions and fine
potentials. Generally one distinguishes by the prefix `fine(ly)'
notions in fine potential theory from those in classical potential
theory on a usual (Euclidean) open set. Large parts of classical
potential theory have been extended to fine potential theory.

The integral representation of positive finely super\-harmonic (=nonnegative)
functions by using Choquet's method of extreme points was studied by the first
named author in \cite{El1}, where it was shown that the cone of positive
super\-harmonic functions equipped with the natural topology has a compact base.
This allowed the present authors
in \cite{EF1} to define the Martin compactification and the Martin boundary of
a fine domain $U$ in $\RR^n$. The Martin compactification $\overline U$ of $U$
was defined by injection of $U$ in a compact base of the cone $\cal S(U)$ of positive
finely super\-harmonic functions on $U$. While the Martin boundary of a usual domain is closed and hence compact, all
we can say in the present setup is that the Martin boundary $\Delta(U)$ of $U$ is a
$G_\delta$ subset of the compact Riesz-Martin space
$\overline U=U\cup\Delta(U)$ endowed with the natural topology. Nevertheless
we can define a suitably measurable Riesz-Martin kernel
$K:U\times\overline U\longrightarrow[0,+\infty]$.
Every function $u\in\cal S(U)$ has an integral representation
$u(x)=\int_{\overline U}K(x,Y)d\mu(Y)$ in terms of a Radon measure $\mu$ on
$\overline U$. This representation is unique if it is required that $\mu$ be
carried by $U\cup\Delta_1(U)$, where $\Delta_1(U)$ denotes the minimal Martin
boundary of $U$, which likewise is a $G_\delta$ in $\overline U$. In this case
of uniqueness we write $\mu=\mu_u$. We show that $u$ is a
fine potential, resp.\ an invariant function, if and only if $\mu_u$ is
carried by $U$, resp.\ by $\Delta(U)$. The invariant functions, likewise
studied in \cite{EF1}, generalize the positive harmonic functions in the
classical Riesz decomposition theorem. Finite valued invariant functions are
the same as positive finely harmonic functions.

There is a notion of minimal thinness of a set $E\subset U$ at a point
$Y\in\Delta_1(U)$, and an associated minimal-fine filter $\cal F(Y)$,
which allowed the authors
in \cite{EF1} to obtain a generalization of the classical Fatou-Na{\"\i}m-Doob
theorem. We showed that,
for any finely super\-harmonic function $u\ge0$ on $U$
and for $\mu_1$-almost every point $Y\in\Delta_1(U)$,
$u(x)$ has the limit $(d\mu_u/d\mu_1)(Y)$ as $x\to Y$ along
the minimal-fine filter $\cal F(Y)$. Here $d\mu_u/d\mu_1$ denotes the
Radon-Nikod{\'y}m derivative of the absolutely continuous component of
$\mu_u$ with respect to the absolutely continuous component of the measure
$\mu_1$ representing the constant function 1, which is finely harmonic and
hence invariant.

In the present continuation of \cite{EF1} we study sweeping on a subset of the
Riesz-Martin space, and the Dirichlet problem at the Martin boundary of $U$.
An important integral representation of swept functions
(Theorem \ref{thm5}) seems to be new even in the case where $U$ is a
Euclidean domain. Furthermore we define the notion of minimal thinness
of a subset of $U$ at a point of $\Delta_1(U)$, and the associated
minimal-fine topology on $\overline U$. This mf-topology
is finer than the natural topology on $\overline U$, and
induces on $U$ the fine topology there.

In a further continuation \cite{EF3} of \cite{EF1} we adapt the PWB method
to the study of the Dirichlet problem at the Martin boundary of the fine
domain $U$.

{\bf Notations}: For a Green domain $\Omega$ in $\RR^n$, $n\ge2$, we denote
by $G_\Omega$ the
Green kernel for $\Omega$.  If $U$ is a fine domain in $\Omega$ we
denote by ${\cal S}(U)$ the convex cone of positive finely super\-harmonic
functions on $U$ in the sense of \cite{F1}. The convex cone of fine potentials
on $U$ (that is, the functions in ${\cal S}(U)$ for which every finely
subharmonic minorant is $\le 0$) is denoted by ${\cal P}(U)$. The cone of
invariant functions on $U$ is denoted by ${\cal H_i}(U)$; it is the
orthogonal band to ${\cal P}(U)$ relative to ${\cal S}(U)$. By $G_U$
we denote the (fine) Green kernel for $U$, cf.\ \cite{F2}, \cite{F4}. If
$A\subset U$ and $f:A\longrightarrow[0,+\infty]$ one denotes by $R{}_f^A$,
resp.\ ${\widehat R}{}_f^A$, the reduced function, resp.\ the swept function,
of $f$ on $A$ relative to $U$, cf.\ \cite[Section 11]{F1}. If $u\in\cal S(U)$
and $A\subset\overline U$ we may write ${\widehat R}{}_u^A$ for
${\widehat R}{}_f$ with $f:=1_Au$.  For any set $A\subset\Omega$ we denote
by $\widetilde A$ the fine closure of $A$ in $\Omega$, and by $b(A)$ the base
of $A$ in $\Omega$,
that is, the set of points of $\Omega$ at which $A$ is not thin,
in other words the set of all fine limit points of $A$ in $\Omega$.

\section{Sweeping on subsets of $\overline U$}\label{sec2}

We shall need an ad hoc concept of a (fine) Perron family. Recall from
\cite[Section 3]{EF1} the continuous affine form $\Phi\ge0$ on $\cal S(U)$
such that the chosen compact base $B$ of the cone $\cal S(U)$ consists of
all $u\in\cal S(U)$ with $\Phi(u)=1$. Cover $\Omega$
by a sequence of Euclidean open balls $B_k$ with closures
$\overline B_k$ contained in $\Omega$. We refer to \cite[Lemma 3.14]{EF1} for
the proof of the following lemma:

\begin{lemma}\label{lemma4.2}
{\rm{(a)}} The mapping $U\ni y\longmapsto G_U(.,y)\in\cal S(U)$
is continuous from $U$ with the fine topology into $\cal S(U)$ with the
natural topology.

{\rm{(b)}} The function $U\ni y\longmapsto\Phi(G_U(.,y))\in \,]0,+\infty[$ is
finely continuous on $U$.

{\rm{(c)}} The sets
$$V_k=\{y\in U:\Phi(G_U(.,y))>1/k\}\cap B_k$$
form a countable cover of \,$U$ by finely open sets which are
relatively naturally compact in $U$.
\end{lemma}

\begin{definition}\label{def4.3} A nonvoid lower directed family
$\cal F\subset\cal S(U)$ is called a (fine) Perron family if
$\widehat R{}_u^{U\setminus V_k}\in\cal F$
for every $k$ and every $u\in\cal F$.
\end{definition}

\begin{theorem}\label{thm4.4} If \,$\cal F\subset\cal S(U)$ is a
Perron family then $\widehat\inf\,\cal F$ is an invariant function, and
$\widehat{\inf}\,\cal F=\inf\,\cal F$ on $\{\inf\,\cal F<+\infty\}$.
\end{theorem}

\begin{proof} Fix $k$. Clearly
$$\widehat\inf\,\cal F
=\widehat\inf\{\widehat R{}_u^{U\setminus V_k}:u\in\cal F\},$$
and the family $\{\widehat R{}_u^{U\setminus V_k}:u\in\cal F\}$ is
lower directed in $U$. By \cite[Lemma 2.4]{EF1} each
$\widehat R{}_u^{U\setminus V_k}$ is invariant in $V_k$, and so is therefore
$\widehat{\inf}\,\cal F_{|V_k}$ according to \cite[Theorem 2.6 (c)]{EF1}.
Consequently, $\widehat{\inf}\,\cal F$ is likewise invariant, by
\cite[Theorem 2.6 (b)]{EF1}. For given $x\in\{\inf\,\cal F<+\infty\}$
(if any), choose $u\in\cal F$ with $u(x)<+\infty$ and an index $k$
so that $x\in V_k$. Then $\widehat R{}_u^{U\setminus V_k}\in\cal F$. The
restriction of $\widehat R{}_u^{U\setminus V_k}$ to the relatively open and hence
finely open subset $U\setminus\overline V{}_k$ of $U$
(cf.\ \cite[Corollary 3.15]{EF1}) is invariant according to
\cite[Lemma 2.4]{EF1}. We have $\widehat R{}_u^{U\setminus\overline V{}_k}=u$ on
$U\setminus\overline V{}_k$ by \cite[Lemma 11.10]{F1} because finely open sets
are subbasic. It follows that $\widehat R{}_u^{U\setminus V_k}=u$, and so $u$ is
invariant on $V_k$, as noted above. So is therefore every minorant of $u$
in $\cal F$, and we conclude from \cite[c), p. 132]{F1} that indeed
$\widehat{\inf}\,\cal F=\inf\,\cal F$ at the given point $x$.
\end{proof}

We are now prepared to study sweeping on $\overline U$, following in part
the classical procedure, cf.\ \cite{Do}, \cite[Section 8.2]{AG},
the main deviations being caused by the non-compactness of $\Delta(U)$.
See also Definition \ref{def4.5a} and Theorem
\ref{thm5.16} below for the analogous and actually equivalent notion of
sweeping relative to the minimal-fine topology on $\overline U$.

\begin{definition}\label{def4.5} Let $A\subset\overline U$. For any
function $u\in {\cal S}(U)$ the reduction of $u$ on $A$ is defined by
$$R{}_u^A=\inf\{v\in\cal S(U): v\ge u\;
\textrm{ on }A\cap U\;\textrm{ and on }W\cap U
\textrm{ for some }W\in\cal W(A)\},$$
where $\cal W(A)$ denotes the family of all open sets
$W\subset\overline U$ with the natural topology such that
$W\supset A\cap \Delta(U)$. The sweeping ${\widehat R}{}_u^A$ of $u$ on $A$
is defined as the regularization of $R_u^A$, that is, the greatest
finely l.s.c.\ minorant of $R_u^A$.
\end{definition}

Thus ${\widehat R}{}_u^A$ is of class $\cal S(U)$.
It is convenient to express $R{}_u^A$ and ${\widehat R}{}_u^A$ in terms of
reduction and sweeping on subsets of $U$, cf.\ \cite[1.III.5]{Do}:
$$R{}_u^A=\inf\{R{}_u^{(A\cup W)\cap U}:W\in\cal W(A)\},$$
$${\widehat R}{}_u^A
={\widehat\inf}\,\{{\widehat R}{}_u^{(A\cup W)\cap U}:W\in\cal W(A)\}.$$
In particular, for any subset $A$ of $U$, the present reduction $R{}_u^A$ and
sweeping ${\widehat R}{}_u^A$ relative to $\overline U$ reduce to the
similarly denoted usual reduction and sweeping on $A$ relative to $U$.
Note that if $A\subset\Delta(U)$ we may replace $A\cup W$ by $W$ in the above
expressions for $R{}_u^A$ and ${\widehat R}{}_u^A$.

By the fundamental convergence theorem \cite[Theorem 11.8]{F1} and the
quasi-Lindel{\"o}f property for finely u.s.c.\ functions
(cf.\ \cite[\S3.9]{F1} for finely l.s.c.\ functions), there is a decreasing
sequence $(W_j)$ of sets $W_j\in\cal W(A)$ (depending on $u$) such that
it suffices to take for $W$ the sets $W_j$, in the above definitions and
alternative expressions.

\begin{remark}\label{remark2.4} If $A\subset\Delta(U)$ then $\cal W(A)$ is the
family of all open sets $W\subset\overline U$ containing $A$, and it then
suffices to take for $W$ a decreasing sequence of open sets $W_j\supset A$
(depending on $u$)
such that $\bigcap_j\overline W_j\subset\overline A$. In fact, $\overline A$
is the intersection of a decreasing sequence of open sets
$V_j\subset\overline U$, and we merely have to replace the above $(W_j)$ by
the decreasing sequence of open sets $W_j\cap V_j\in\cal W(A)$ whose
intersection clearly is contained in $\overline A$. If $A$ is a compact
subset of $\Delta(U)$ we may therefore take $W_j=V_j$ (independently of $u$).
\end{remark}

\begin{prop}\label{prop4.6} Let $A$ and $B$ be two subsets of
$\overline U$ and let $u,v\in {\cal S}(U)$ and $0<\alpha<+\infty$. Then

{\rm{1.}} ${\widehat R}{}_u^{A\cup B}\le{\widehat R}{}_u^A+{\widehat R}{}_u^B$.

{\rm{2.}} If $A\subset B$ then ${\widehat R}{}_u^A\le {\widehat R}{}_u^B$.

{\rm{3.}} If $0$ times $+\infty$ is defined to be $0$ then
${\widehat R}{}_{\alpha u}^A = \alpha{\widehat R}{}_u^A.$

{\rm{4.}} ${\widehat R}{}_{u+v}^A= {\widehat R}{}_u^A+{\widehat R}{}_v^A$.

{\rm{5.}} For any decreasing sequence of functions $u_j\in\cal S(U)$ we have
$\widehat\inf_j\,\widehat R{}_{u_j}^A
=\widehat R{}_{{\widehat\inf}_j\,u_j}^A$.

{\rm{6.}} If  $A\subset B$ then
${\widehat R}{}_{{\widehat R}{}_u^A}^B=
 {\widehat R}{}_{{\widehat R}{}_u^B}^A={\widehat R}{}_u^A$.
\end{prop}

\begin{proof} Property 1.\ is established just as in
\cite[(4.1), p.\ 39]{Do} (with $v=0$): For $u_A\in\cal S(U)$ with
$u_A\ge u$ on $A\cup W_A$ for some $W_A\in\cal W(A)$
and analogous $u_B,W_B$ we have $u_A+u_B\in\cal S(U)$ and $u_A+u_B\ge u$ on
$A\cup B\cup W_A\cup W_B$ with $W_A\cup W_B\in\cal W(A\cup B)$. This
implies $R{}_u^{A\cup B}\le R{}_u^A+R{}_u^B$. The asserted
inequality therefore holds quasieverywhere and hence everywhere on $U$, by
fine continuity.
Property 2.\ follows from $\cal W(A)\supset\cal W(B)$. Property 3.\ follows
from
$\widehat R{}_{\alpha u}^{(A\cup W)\cap U}=\alpha\widehat R{}_u^{(A\cup W)\cap U}$
by taking $\widehat{\inf}$ over $W\in\cal W(A)$.
As to 4.\ we have for any $W\in\cal W(A)$
$\widehat R{}_{u+v}^{(A\cup W)\cap U}
=\widehat R{}_u^{(A\cup W)\cap U}+\widehat R{}_v^{(A\cup W)\cap U}$,
whence the asserted equation by taking the natural limits of the decreasing
nets on $\cal S(U)$ in question as the index $W$ ranges over the lower directed
family $\cal W(A)$, cf.\ \cite[Theorem 2.9]{EF1}.

Concerning 5., according to \cite[Lemmas 2.1 and 2.3]{EF1}, $u_j$ is Euclidean
Borel measurable and
$$
R_{u_j}^{(A\cup W)\cap U}(x)
=\int_Uu_j\,d\eps_x^{\complement(U\setminus(A\cup W))}\le u_j(x)\le u_1(x)
$$
(complements relative to $\Omega$). For given $x\in U$ with $u_1(x)<+\infty$
and $W\in\cal W(A)$ consider the equalities
\begin{eqnarray*}\inf_jR{}_{u_j}^{(A\cup W)\cap U}(x)\!\!\!
&=&\!\!\!\inf_j\int_Uu_j\,d\eps_x^{\complement(U\setminus(A\cup W))}
=\int_U\inf_ju_j\,d\eps_x^{\complement(U\setminus(A\cup W))}\\
&=&\!\!\!\int_U\widehat{\inf_j}\,u_j\,d\eps_x^{\complement(U\setminus(A\cup W))}
=R_{\widehat\inf_ju_j}^{(A\cup W)\cap U}(x).
\end{eqnarray*}
The first and the last equalities hold by \cite[Lemma 2.3]{EF1}. The second
equality is obvious (Lebesgue), the integrals being finite by hypothesis.
The third equality holds if $\widehat\inf_j u_j(x)=\inf_ju_j(x)$, for either
$x\in U\cap b((A\cup W)\cap U)$, and then
$\eps_x^{\complement(U\setminus(A\cup W))}=\eps_x$; or else
$x\in U\setminus b((A\cup W)\cap U)$, and then
$\eps_x^{\complement(U\setminus(A\cup W))}$
does not charge the polar set $\{\inf_ju_j\ne\widehat{\inf_j}\,u_j\}$.
The resulting equality in the above display thus
holds q.e.\ for $x\in U$, and hence also everywhere on $U$ after finely
l.s.c.\ regularization of both members.

Property 6.\ is known for $A,B\subset U$. For general
$A,B\subset\overline U$, say with $A\subset B$,
the first and the second member of the equalities in 6.\ lie between
${\widehat R}{}_{{\widehat R}{}_u^A}^A$ and $\widehat R{}_u^A$ in view of 2.,
and it therefore suffices to consider the case where $A=B$.
For given $W\in\cal W(A)$ consider a decreasing sequence
$(W_j)\subset\cal W(A)$ such that ${\widehat R}{}_u^A
=\widehat{\inf}_j\,{\widehat R}{}_u^{(A\cup W_j)\cap U}.$
Replacing $W_j$ by $W_j\cap W$ we achieve that $W_j\subset W$, and hence
$$
\widehat R{}_{\widehat R{}_u^{(A\cup W_j)\cap U}}^{(A\cup W)\cap U}
=\widehat R{}_u^{(A\cup W_j)\cap U}.
$$
According to 5.\ this implies 6.\ by taking
${\widehat\inf}_j$ and next taking
$\widehat{\inf}_{W\in\cal W(A)}$.
\end{proof}

\begin{prop}\label{prop4.7} Let $u\in {\cal S}(U)$.
For any subset $A$ of \,$\Delta(U)$
the function  ${\widehat R}{}_u^A$ is invariant, and we have
${\widehat R}{}_u^A\preccurlyeq u$.
\end{prop}

\begin{proof} Consider the family
$$\cal F:=\{\widehat R{}_u^{W\cap U}:W\in\cal W(A)\}.$$
Clearly, $\cal F$ is lower directed. Consider the compact sets
$A_{kl}\subset U$ in the proof of \cite[Proposition 3.10]{EF1}. Fix $k$ and $V_k$
from Lemma \ref{lemma4.2}. In view of that lemma and the text preceding it,
$\overline U\setminus A_{kk}$ is open in
$\overline U$ and contains $\Delta(U)$. In Definition \ref{def4.5} of
${\widehat R}{}_u^A$ it therefore suffices to
consider open sets $W\supset A$
such that $W\subset\overline U\setminus A_{kk}$, whereby
$W\cap U\subset U\setminus A_{kk}\subset U\setminus V_k$. By 6.\ in Proposition
\ref{prop4.6} we then have
${\widehat R}{}_{{\widehat R}{}_u^{W\cap U}}^{U\setminus V_k}=\widehat R{}_u^{W\cap U}$.
The lower directed family
$\cal F:=\{\widehat R{}_u^{W\cap U}:W\in\cal W(A)\}$ is therefore a
Perron family in the sense of Definition \ref{def4.3}. By Definition
\ref{def4.5} we have ${\widehat R}{}_u^A=\widehat\inf\,\cal F$, and it
therefore follows by Theorem \ref{thm4.4} that ${\widehat R}{}_u^A$ indeed
is invariant. Consequently, ${\widehat R}{}_u^A\preccurlyeq u$
in view of \cite[Lemma 2.2]{EF1}.
\end{proof}

\begin{prop}\label{prop4.7a} Let $u\in\cal S(U)$.
{\rm{(a)}} For any increasing sequence $(A_j)$ of subsets of \,$\overline U$
we have ${\widehat R}{}_u^{\bigcup_jA_j}=\sup_j{\widehat R}{}_u^{A_j}.$

{\rm{(b)}} For any sequence $(A_j)$ of subsets of \,$\overline U$ we have
${\widehat R}{}_u^{\bigcup_jA_j}\le\sum_j{\widehat R}{}_u^{A_j}.$
\end{prop}

\begin{proof}
For (a) we proceed much as in \cite[p. 74, Proof of (e)]{Do} (where $U$ is a
Euclidean Green domain).
Writing $A=\bigcup_jA_j$ and $v=\sup_j\widehat R{}_u^{A_j}$ the inequality
$v\le\widehat R{}_u^A$ is obvious. For the opposite inequality we shall also
consider $\widehat R{}_u^{A_j\cap\Delta(U)}$. Consider a point
$x\in U$ for which $u(x)<+\infty$ and
$\widehat R{}_u^{A_j\cap\Delta(U)}(x)=R{}_u^{A_j\cap\Delta(U)}(x)$.
For any integer $j>0$ there exists
$W_j\in\cal W(A_j\cap\Delta(U))=\cal W(A_j)$ and
$v_j\in\cal S(U)$ such that $v_j\ge u$ on $W_j\cap U$ and
$$
v_j(x)\le R{}_u^{A_j\cap\Delta(U)}(x)+2^{-j}
=\widehat R{}_u^{A_j\cap\Delta(U)}(x)+2^{-j}.
$$
The swept function $\widehat R{}_u^{A_j\cap\Delta(U)}$
is invariant by Proposition \ref{prop4.7}, and
$\widehat R{}_u^{A_j\cap\Delta(U)}\le\widehat R{}_u^{W_j\cap U}
\le R{}_u^{W_j\cap U}\le v_j$.
Hence $\widehat R{}_u^{A_j\cap\Delta(U)}\preccurlyeq v_j$.
We show that for any integer $k>0$ the function
\begin{eqnarray}u'_k\!\!\!
&:=&\!\!\!v+\sum_{j\ge k}(v_j-\widehat R{}_u^{A_j\cap\Delta(U)})
\end{eqnarray}
is of class $\cal S(U)$. In the first place, each term in the sum is of class
$\cal S(U)$.
Because $v_j$ is finely continuous and $R{}_u^{A_j\cap\Delta(U)}$ is finely
u.s.c.\ there is a fine neighborhood $V$ of $x$ with Euclidean
compact closure $\overline V$ in $\Omega$ contained in $U$ and such that
$v_j\le R{}_u^{A_j\cap\Delta(U)}+2^{1-j}$ on $V$ and hence on $\overline V$, by fine
contimuity.  We may further arrange that $u$ is bounded on $\overline V$
and that $\widehat R{}_u^{A_j\cap\Delta(U)}=\widehat R{}_u^{W_j\cap U}$ on
$\overline V$. Then
$$
\int(v_j-\widehat R{}_u^{A_j\cap\Delta(U)})d\eps_x^{\Omega\setminus V}
\le v_j(x)-\widehat R{}_u^{A_j\cap\Delta(U)}(x)\le2^{1-j},
$$
since $\eps_x^{\Omega\setminus V}$ is carried by $\overline V$ and does not
charge any polar set. See also\cite[Section 8.4]{F1}. It follows that
the finely hyper\-harmonic sum in (2.1) is of class $\cal S(U)$,
having a finite integral with respect to $\eps_x^{\Omega\setminus V}$.
For any
$W\in\cal W(A_j)$ we have $\widehat R{}_u^{(A_j\cup W)\cap U}=u$ q.e.\ on
$(A_j\cup W)\cap U$, in particular q.e.\ on $A_j\cap U$. By Definition
\ref{def4.5} we have $\widehat R{}_u^{A_j}=u$
q.e.\ on $A_j\cap U$ (because it suffices to consider a suitable sequence
of sets $W$). It follows that $v=u$ q.e.\ on each $A_j\cap U$ and hence also
q.e.\ on $A\cap U$. Choose a  super\-harmonic function $s>0$ on $\Omega$ such
that $s(y)=+\infty$ for every $y$ in the polar set
$\{y\in A\cap U:v(y)\ne u(y)\}\cup\bigcup_{j>0}
\{y\in A\cap U:\widehat R{}_u^{A_j\cap\Delta(U)}\ne R{}_u^{A_j\cap\Delta(U)}.$
For any $\delta>0$ we then have $u'_k+\delta s\ge v+\delta s\ge u$ on $A\cap U$.
Because $v\ge\widehat R{}_u^{A_j}\ge\widehat R{}_u^{A_j\cap\Delta(U)}$
we obtain for $j\ge k$
$$
u'_k\ge\widehat R{}_u^{A_j\cap\Delta(U)}+(v_j-\widehat R{}_u^{A_j\cap\Delta(U)})=v_j,
$$
and hence $u'_k\ge v_j\ge u$ on $W_j\cap U$ for $j\ge k$. Altogether,
$u'_k+\delta s\ge u$ on $A\cap U$ and on $W_j\cap U$. It follows by
Definition \ref{def4.5} that
$u'_k+\delta s\ge\widehat R{}_u^A$, and hence for $\delta\to0$ that
$u'_k\ge\widehat R{}_u^A$ q.e.\ and actually everywhere on $U$.
But $u'_k\searrow v$ q.e.\ (namely at each point where $u'_1$ is finite).
Consequently $v\ge\widehat R{}_u^A$ q.e.\ on $U$ and so indeed everywhere
on $U$.

(b) is easily deduced from (a) applied with $A_j$ replaced by
$A_1\cup\ldots\cup A_j$, in view of 1.\ in Proposition \ref{prop4.6}
(extended to finite unions). There is also a simple direct proof, cf.\
\cite[Lemma 8.2.2 (i)]{AG} for the case of a Euclidean Green domain $U$.
\end{proof}

\begin{prop}\label{prop4.8} For any $A\subset \overline U$ we have
$\widehat R{}_u^A=\widehat R{}_u^{A\cap \Delta(U)}+\widehat R{}_v^{A\cap U}$,
where $v:=u-\widehat R{}_u^{A\cap \Delta(U)}$.
\end{prop}

\begin{proof} Let $s\in\cal S(U)$, $s\ge u$ on $A\cap U$ and on a
neighborhood of $A\cap\Delta(U)$. Then
$s\ge\widehat R{}_u^{A\cap\Delta(U)}$, which is invariant by Proposition
\ref{prop4.7}, and so
$ \widehat R{}_u^{A\cap \Delta(U)} \preccurlyeq s$. Furthermore,
$s-\widehat R{}_u^{A\cap\Delta(U)}\ge u-\widehat R{}_u^{A\cap\Delta(U)}=v$ on $A\cap U$,
It follows that
$s- \widehat R{}_u^{A\cap \Delta(U)}\ge \widehat R{}_v^{A\cap U}$, and so
$s\ge  \widehat R{}_u^{A\cap \Delta(U)}+\widehat R{}_v^{A\cap U}$. This shows that
$\widehat R{}_u^A\ge  \widehat R{}_u^{A\cap \Delta(U)}+\widehat R{}_v^{A\cap U}$.
For the opposite inequality let $w\in S(U)$, $w\ge u$ on $A\cap U$ and on
a neighborhood of $A\cap\Delta(U)$.
Then $w\ge\widehat R{}_u^{A\cap\Delta(U)}$ on $A\cap U$, and since
$\widehat R{}_u^{A\cap \Delta(U)}$ is invariant as noted above, we have
$w-\widehat R{}_u^{A\cap\Delta(U)}\in\cal S(U)$. By hypothesis this function
majorizes $v$ on $A\cap U$, and we therefore get
$w\ge\widehat R{}_u^{A\cap \Delta(U)}+\widehat R{}_v^{A\cap U}$ . By varying $w$ this
leads to the remaining inequality
${\widehat R}{}_u^A\ge\widehat R{}_u^{A\cap \Delta(U)}+\widehat R{}_v^{A\cap U}$.
\end{proof}

For any (positive Radon) measure $\mu$ on $\overline U$ we write for brevity
$K\mu=\int_{\overline U}K(.,Y)d\mu(Y)$.
We say that a measure $\mu$ represents a function $u\in\cal S(U)$ if $u=K\mu$.

\begin{cor}\label{cor}
Let $H$ be a compact subset of $\overline U$ and $\mu$ a Radon measure on
$\overline U$ carried by $H$. Then $K\mu$ is invariant on $U\setminus H$.
\end{cor}

\begin{proof} By Proposition \ref{prop4.8} there is a function $v\in\cal S(U)$
such that ${\widehat R}{}_{K\mu}^H
={\widehat R}{}_{K\mu}^{H\cap\Delta(U)}+{\widehat R}{}_v^{H\cap U}$. By Proposition
\ref{prop4.7} the former term on the right is invariant, and the latter term
is invariant on $U\setminus H$ according to \cite[Lemma 2.4]{EF1}.
\end{proof}

\begin{cor}\label{cor4.11} Let $A\subset\overline U$ and $u\in\cal S(U)$.
Then ${\widehat R}{}_u^A$ from Definition
\ref{def4.5} is invariant on $U\setminus{\overline A}$.
\end{cor}

\begin{proof} For $A\subset U$ this follows from \cite[Lemma 2.4]{EF1},
and for $A\subset\Delta(U)$ it follows from Proposition \ref{prop4.7}.
For general $A\subset\overline U$ it therefore follows right away
by application of Proposition \ref{prop4.8}.
\end{proof}

For a (positive) Borel measure $\mu$ on $U$ we denote by
$\mu_.$ and $\mu^.$ the inner and the outer $\mu$-measure, respectively .

\begin{prop}\label{prop2} Let $p=G_U\mu$ be a fine potential on $U$
and let $V$ be a finely open subset of $U$. Then we have the bi-implications
$$
p_{|V}{\text{\rm{ invariant}}}\;\Longleftrightarrow\;\mu_.(V)=0
\;\Longleftrightarrow\;\mu^.(V)=0.
$$
\end{prop}

\begin{proof} Suppose first that $\mu_.(V)=0$. For any regular finely open
set $W$ with $\widetilde W\subset V$ the part of $\mu$ on the $K_\sigma$ set
$W$ equals $0$, and hence $G_W\mu\equiv0$. It follows by
\cite[Lemma 2.6]{F4} that
$$
p=\int p\,d\eps_x^{\Omega\setminus W}
$$
on $U$, in particular on $V$. Fix a point $x_0\in V$. The finely open sets
$W_j:=\{x\in V:G_U(x,x_0)>\frac1j\}$ form a countable cover of $V$ such
that $\widetilde W{}_j\subset V$. It therefore follows by
\cite[Theorem 4.4]{F4} (with $U$ replaced by $V$ and $s$ by $p_{|V}$) that
$p_{|V}$ indeed is invariant.

Conversely, suppose that $p_{V}$ is invariant, and let us prove that
$\mu^.(V)=0$. Under the extra hypothesis
that $\widetilde V\subset U$ it now follows by \cite[Lemma 2.6]{F4} that
$$
p=G_V\mu+\widehat R{}_p^{U\setminus V}
$$
on $U$. By \cite[Lemma 2.4]{F4} the latter term on the right is invariant
on $V$, and so is therefore the difference $p=G_V\mu$, which however is a
fine potential on $V$, and so $G_V\mu=0$ on $V$, that is, $\mu(r(V))=0$,
whence $\mu^.(V)=0$. Without the above extra hypothesis that
$\widetilde V\subset U$ we cover $V$ by a sequence of finely open sets
$$
W_j:=\{x\in V:G_{r(V)}(x,x_0)>\frac1j\}
\subset\{x\in r(V):G_{r(V)}(x,x_0)\ge\frac1j\}.
$$
The last set is finely closed subset of $U$, and so
$\widetilde W{}_j\subset U$. As shown above, it follows that
$\mu^.(W_j)=0$, and hence indeed $\mu^.(V)=0$.
\end{proof}

\begin{prop}\label{prop4.14} Let $A\subset\overline U$ and $u\in\cal S(U)$.
Then there exists a measure on $\overline U$ representing
$\widehat R{}_u^A$ and carried by $\overline A$.
\end{prop}

\begin{proof} We may suppose that $\widehat R_u^A\ne 0$, in particular $u>0$,
the case $\widehat R_u^A=0$  being trivial.  For any probability measure
$\nu$ on $B$ we denote in this proof by $b(\nu)$ the barycenter of $\nu$.

Suppose first that $A\subset U$. Let $p$ be a fine potential $>0$ on $U$.
For any natural number $k$ there exists a non-zero Radon measure
$\sigma_k$ on $\overline U$ representing
the fine potential $\widehat R_{u\wedge kp}^A>0$ on $U$, and $\sigma_k$ is
carried by $U$ according to \cite[Corollary 3.25]{EF1}.
In view of the first paragraph of \cite[Section 3]{EF1} we have
$$
\widehat R_{u\wedge kp}^A=K\sigma_k
=\int_UK(.,y)d\sigma_k(y)=\int_UG_U(.,y)d\tau_k(y)
$$
with $d\tau_k(y)=\Phi(G_U(.,y))^{-1}d\sigma_k(y)$.
Here we use that the finite non-zero function
$y\longmapsto\Phi(G_U(.,y))$ on $U$ is finely continuous by
Lemma \ref{lemma4.2} (b) and hence Borel measurable by \cite[Lemma 2.1]{EF1}.
Thus there is indeed a non-zero Borel measure $\tau_k$ on $U$ as stated.
By Corollary \ref{cor4.11} $\widehat R_{u\wedge kp}^A$ is invariant on
$U\setminus\overline A$,
and hence $\tau_k$ is carried by $\overline A$ according to Proposition
\ref{prop2}. It follows that $\sigma_k$ likewise is carried by $\overline A$.

Consider for each $k$ the probability measure $\nu_k$ on the chosen compact
base $B$ of the cone $\cal S(U)$, defined by
$\nu_k(E)=\sigma_k(E\cap U)/\sigma_k(U)$ for any Borel subset $E$ of
$\overline U$. Clearly, $\nu_k$ is carried by $\overline A$ along with
$\sigma_k$. The sequence
$(\nu_k)$ has a subsequence $(\nu_{k_j})$ which converges vaguely to a
probability measure $\nu$ on $\overline U$, necessarily carried by
$\overline A$. On the other hand,
${\widehat R}{}_{u\wedge kp}^A\to{\widehat R}{}_u^A$ pointwise and increasingly
for $k\to+\infty$. It follows by \cite[Theorem 2.10]{EF1} that
${\widehat R}{}_{u\wedge kp}^A\to{\widehat R}{}_u^A$
in the natural topology on $\cal S(U)$ as $k\to+\infty$, and hence
$\Phi({\widehat R}{}_{u\wedge kp}^A)\to\Phi({\widehat R}{}_u^A)\in\,]0,+\infty[$
because $\Phi$ is naturally continuous on $\cal S(U)$. Identifying as usual
$\nu_k$ and $\nu$ with probability measures on $B$ we infer that
$$\frac1{\Phi({\widehat R}{}_u^A)}{\widehat R}{}_u^A
=\lim_{j\to\infty}\frac1{\Phi({\widehat R}{}_{u\wedge k_jp}^A)}{\widehat R}{}_{u\wedge k_jp}^A
=\lim_{j\to\infty}b(\nu_{k_j})=b(\nu)=K\nu.$$
Hence ${\widehat R}{}_u^A= K\mu$, where $\mu:=\Phi({\widehat R}{}_u^A)\nu$
(now again considered as a measure on $\overline U$) is carried
by $\overline A$ along with $\nu$.

Next, let $A\subset\Delta(U)$.
According to Remark \ref{remark2.4} there is
a decreasing sequence of open sets $W_j$ (depending on $u$) such that
$A\subset\bigcap_jW_j\subset\bigcap_j{\overline W}_j\subset\overline A$ and
${\widehat R}{}_u^A={\widehat{\inf}}_j{\widehat R}{}_u^{W_j\cap U}=
\lim_j{\widehat R}{}_u^{W_j\cap U}$ (natural limit, again by
\cite[Theorem 2.10]{EF1}). There is a sequence of reals
$\alpha_j>0$ and a real $\alpha>0$ such that
$\alpha_j\widehat R_u^{W_j\cap U}\in B$ and $\alpha{\widehat R}{}_u^A\in B$.
The sequence $(\alpha_j)$  converges to $\alpha$ because the sequence
$({\widehat R}{}_u^{W_j\cap U})$ converges naturally to ${\widehat R}{}_u^A$. For
any index $j$ there exists, as shown in the preceding paragraph, a probability
measure $\mu_j$ on $B$ with the barycenter
$\alpha_j{\widehat R}{}_u^{W_j\cap U}$ such that $\mu_j$ (when viewed as a
measure on $\overline U$) is carried by ${\overline W}_j$.
After passing to a subsequence we may
suppose that $\mu_j$ converges to a probability measure $\mu$ on $B$ which
(again when viewed as a measure on $\overline U$) necessarily is carried by
$\bigcap_j{\overline W}_j\subset\overline A$.
The sequence $(b(\mu_j))=(\alpha_j{\widehat R}{}_u^{W_j})$ of barycenters
of the $\mu_j$ therefore converges to the barycenter $b(\mu)$ of $\mu$, whence
$K\mu=b(\mu)=\alpha{\widehat R}{}_u^A$, and $\widehat R_u^A$ is represented by
the measure $\frac{1}{\alpha}\mu$ carried by $\overline A$.

In the general case where just $A\subset \overline U$ we have by Proposition
\ref{prop4.8}
${\widehat R}{}_u^A={\widehat R}{}_u^{A\cap \Delta(U)}+{\widehat R}{}_v^{A\cap U}$,
where $v\in\cal S(U)$. As shown in the third paragraph of the present proof
there exists a measure $\mu_1$ on
$\overline U$ representing ${\widehat R}{}_v^{A\cap U}$ and carried by
$\overline A$. And as shown in the preceding paragraph there exists
a measure $\mu_2$ on $\overline U$ representing $\widehat R_u^{A\cap \Delta(U)}$
and likewise carried by $\overline A$. The measure $\mu=\mu_1+\mu_2$ therefore
represents ${\widehat R}{}_u^A$ and is carried by $\overline A$.
\end{proof}


\begin{prop}\label{prop4.16}
Let $A\subset\overline U$. Then

{\rm{(i)}} If $A\subset\Delta(U)$ then ${\widehat R}{}_p^A=0$ for any
$p\in\cal P(U)$.

{\rm{(ii)}} If $A\subset\Delta(U)$ and $Y\in U\cup\Delta_1(U)$ then either
${\widehat R}{}_{K(.,Y)}^A=0$ or ${\widehat R}{}_{K(.,Y)}^A=K(.,Y)$.
If moreover $Y\notin A$ then ${\widehat R}{}_{K(.,Y)}^A=0$.

{\rm{(iii)}} If $Y\in\Delta_1(U)\setminus\overline A$ then
${\widehat R}{}_{K(.,Y)}^A\ne K(.,Y)$.
\end{prop}

\begin{proof} It follows from Proposition \ref{prop4.7} that
${\widehat R}{}_p^A$ and ${\widehat R}{}_{K(.,Y)}^A$
are invariant. This establishes (i) because ${\widehat R}{}_p^A$ is also
a fine potential (along with $p$).

For the former assertion (ii) we have
${\widehat R}{}_{K(.,Y)}^A\preccurlyeq K(.,Y)$,
again by Proposition \ref{prop4.7}, and since
$K(.,Y)$ is extreme there is a constant $c\ge 0$ such that
${\widehat R}{}_{K(.,Y)}^A=cK(.,Y)$. Hence it follows from 6.\ in Proposition
\ref{prop4.6} with $A=B$ that $c=0$ or $c=1$.

For the latter assertion (ii) suppose first that $Y\notin{\overline A}$
(the natural closure of $A$ in $\overline U$). Suppose by contradiction that
${\widehat R}{}_{K(.,Y)}^A=K(.,Y)$.
According to Proposition \ref{prop4.14} there exists a measure
$\lambda$ on $\overline U$ carried by $\overline A$ such that
\begin{eqnarray} {\widehat R}{}_{K(.,Y)}^A
=\int_{\overline U}K(.,Z)d\lambda(Z).
\end{eqnarray}
It follows that $K(.,Y)=\int_{\overline U}K(.,Z)d\lambda(Z),$
and so $\lambda$ is a probability measure.
Denote $\mu$ the probability measure on $B$ corresponding to $\lambda$ under
the identification of $B$ with $\{K(.,Z):Z\in\overline U\}$.
Then $\mu$ has the barycenter $K(.,Y)$. Since $K(.,Y)$ is an extreme point
of $B$ we infer by \cite[Corollary I.2.4, p.15]{Al} that $\mu=\eps_Y$, and hence
$Y\in\overline A$, which is contradictory. Thus
${\widehat R}{}_{K(.,Y)}^A\ne K(.,Y)$,
and consequently ${\widehat R}{}_{K(.,Y)}^A=0$  according to the former
assertion (ii) above.

It remains to consider the case where we just have $Y\notin A$.
In this case $A$ may be written as the union of an increasing
sequence of subsets $A_j$ of $A$ with
$Y\notin {\overline A_j}$ for any $j$. By Proposition
\ref{prop4.7a} (b) we then have
${\widehat R}{}_{K(.,Y)}^A\le \sum_j{\widehat R}{}_{K(.,Y)}^{A_j}=0$,
and hence ${\widehat R}{}_{K(.,Y)}^A=0$, as claimed.

For (iii), suppose by contradiction that ${\widehat R}{}_{K(.,Y)}^A=K(.,Y)$.
Again, there exists by Proposition \ref{prop4.14} a probability measure
$\lambda$ on $\overline U$ carried by $\overline A$ such that (2.2) holds,
and hence $Y\in\overline A$, which is contradictory.
\end{proof}

Actually, in Proposition \ref{prop4.16} (ii),
if $Y\in A$ and hence $Y\notin U$ then $Y\in\Delta_1(U)$, and it follows
that ${\widehat R}{}_{K(.,Y)}^A=K(.,Y)$, see Proposition \ref{prop4.18} below.

The following result extends 4.\ in Proposition $\ref{prop4.6}$ to infinite
sums.

\begin{prop}\label{prop4.17} Let $A\subset\overline U$. Let $(\mu_j)$ be a
sequence of measures on $\overline U$ such that $\sum_j\int d\mu_j<+\infty$,
and let $\mu=\sum_j\mu_j$. Then
$$
\widehat R_{K\mu}^A=\sum_j\widehat R{}_{K\mu_j}^A.
$$
\end{prop}

\begin{proof} If $A\subset U$ we have indeed by \cite[Lemma 2.3]{EF1}
for $x\in U$
\begin{eqnarray*}R_{K\mu}^A(x)\!\!\!
&=&\!\!\!\int K\mu\,d\eps_x^{A\cup(\Omega\setminus U)}
=\int\Bigl(\int K(.,Y)\,d\eps_x^{A\cup(\Omega\setminus U)}\Bigr)d\mu(Y)\\
&=&\!\!\!\sum_j\int\Bigl(\int K(.,Y)\,d\eps_x^{A\cup(\Omega\setminus U)}\Bigr)d\mu_j(Y)
=\sum_j\int K\mu_j\,d\eps_x^{A\cup(\Omega\setminus U)}\\
&=&\!\!\!\sum_jR_{K\mu_j}^A(x),
\end{eqnarray*}
the applications of Fubini's theorem being justified by the conclusion of
\cite[Remark 3.4]{EF1}.
It only remains to perform the l.s.c.\ regularization of both members of
the resulting equation.

Next, let $A\subset\Delta(U)$. The inequality `$\ge$' being trivial we may
suppose that the finely hyper\-harmonic function given by the right hand side
of the asserted equation is of class $\cal S(U)$, viz.\ $\not\equiv+\infty$.
Suppose first that $\widehat R{}_{K\mu}^A=K\mu$. For integers $k>0$ we have by
4. in Proposition $\ref{prop4.6}$ (extended to sums of finitely many measures)
$$
\widehat R{}_{K\mu}^A
=\sum_{j\le k}\widehat R_{K\mu_j}^A+\widehat R{}_{K(\sum_{j>k}\mu_j)}^A.
$$
Since $\sum_jK\mu_j=K\mu<+\infty$ q.e.\ we have
$$
\widehat R{}_{K(\sum_{j>k}\mu_j)}^A\le K\sum_{j>k}\mu_j=\sum_{j>k}K\mu_j\searrow0
$$
q.e.\ as $k\to\infty$. We thus have
$\widehat R{}_{K\mu}^A\ge\sum_{j>0}\widehat R_{K\mu_j}^A$ with equality q.e.,
and indeed everywhere, both members of the inequality being of class
$\cal S(U)$.

Without the temporary hypothesis $\widehat R{}_{K\mu}^A=K\mu$ we have
$\widehat R{}_{K\mu}^A\preccurlyeq K\mu$ because $\widehat R{}_{K\mu}^A$ is
invariant according to Proposition \ref{prop4.7}. Thus there exists a
measure $\nu$ on $\overline U$ with $\nu\le\mu$ and hence
$K\nu\in\cal S(U)$ such that $K\mu=\widehat R_{K\mu}^A+K\nu$. After
sweeping on $A$ while invoking 6.\ in Proposition \ref{prop4.6} we obtain
$\widehat R_{K\nu}^A=0$. Thus
$K(\mu-\nu)=\widehat R{}_{K(\mu-\nu)}^A=\widehat R{}_{K\mu}^A$ .
Similarly, $K\mu_j=\widehat R_{K\mu_j}^A+K\nu_j$ with $\nu_j\le\mu_j$ and
$K(\mu_j-\nu_j)=\widehat R{}_{K\mu_j}^A$. As shown above it follows that
$\widehat R{}_{K(\mu-\nu)}^A =\sum_j\widehat R{}_{K(\mu_j-\nu_j)}^A$ and hence
$$
\widehat R{}_{K\mu}^A=\widehat R{}_{K\nu}^A+\sum_j\widehat R{}_{K\mu_j}^A
=\sum_j\widehat R{}_{K\mu_j}^A
$$
because $\widehat R{}_{K\nu}^A=0$.
The general case $A\subset\overline U$ follows immediately by application
of Proposition \ref{prop4.8}.
\end{proof}

\begin{remark}\label{remark4.19} Even in the classical case $U=\Omega$,
sweeping (and reduction)
of a function $u\in\cal S(U)$ on an arbitrary set $A\subset\Delta(U)$ lacks
the following two properties, valid when $A\subset U$.
Fix a point $Y\in A\cap\Delta_1(U)$ and note that
${\widehat R}{}_{K(.,Y)}^A=K(.,Y)$ by Proposition \ref{prop4.18} below.
In the classical case, $K(.,Y)\wedge c$ is a potential
(hence a fine potential) for any constant $c>0$, as noted
in \cite[Observation, p.\ 74]{Do} for the purpose of showing that the
following Property 1.\ fails when $A=\Delta(U)$:

1. For any increasing sequence of functions $u_j\in\cal S(U)$ with pointwise
supremum $u\in\cal S(U)$, we should have
${\widehat R}{}_u^A=\sup_j{\widehat R}{}_{u_j}^A.$
This holds when $A\subset U$, by \cite[Theorem 11.12]{F1}, but fails
(classically) for $A=\Delta(U)$ and $u_j=K(.,Y)\wedge j$ in view of the above.
It does hold, however, for any sequence
$(u_j)\subset\cal S(U)$ which is increasing in the specific order;
this is a reformulation of Proposition \ref{prop4.17}
above.

2. For any $x\in U$, the affine function $u\longmapsto{\widehat R}{}_u^A(x)$
on $\cal S(U)$ should be (naturally) l.s.c.
For the proof that this holds for $A\subset U$ we may assume that $A$ is a
base relative to $U$, and hence $\widehat R{}_u^A=R_u^A$ and $\eps_x^A$ is
carried by $A$ for any $u\in\cal S(U)$. Consider a sequence of functions
$u_j\in\cal S(U)$ converging (naturally) to $u\in\cal S(U)$. Then
\begin{eqnarray*} R_u^A(x)\!\!\!
&=&\!\!\!\int_Uu\,d\eps_x^{A\cup(\Omega\setminus U)}
=\int_U\underset{j}{\lim{\widehat\inf}}\,u_j\,d\eps_x^{A\cup(\Omega\setminus U)}\\
&\le&\!\!\!\int_U\underset{j}{\lim\inf}\,u_j\,d\eps_x^{A\cup(\Omega\setminus U)}
\le\underset{j}{\lim\inf}\,\int_Uu_j\,d\eps_x^{A\cup(\Omega\setminus U)}
=\underset{j}{\lim\inf}\,R_{u_j}^A(x)
\end{eqnarray*}
by \cite[Lemmas 2.1 and 2.3, and Theorem 2.10]{EF1}, using Fatou's lemma; and it only remains to regularize.
But Property 2.\ fails (classically) for $A=\Delta(U)$ and
$u_j=K(.,Y)\wedge j$, hence $u=K(.,Y)$, in view of the above.
\end{remark}

\section{Minimal thinness and the minimal-fine topology}\label{sec3}

The following lemma extends \cite[Lemma 4.2]{EF1}, in which $E\subset U$.

\begin{lemma}\label{lemma5.1}
For any set $E\subset\overline U$ and any point \,$Y\in\Delta_1(U)$ we have
${\widehat R}{}_{K(.,Y)}^E\ne K(.,Y)$ if and only if
$\widehat R{}_{K(.,Y)}^E\in{\cal P}(U)$ (the fine potentials on $U$).
 \end{lemma}

\begin{proof} If $\widehat R{}_{K(.,Y)}^E$ is a fine potential then
$\widehat R{}_{K(.,Y)}^E\ne K(.,Y)$ because $K(.,Y)$ is invariant. Conversely,
suppose that $\widehat R{}_{K(.,Y)}^E\ne K(.,Y)$, and write
$\widehat R{}_{K(.,Y)}^E=p+h$ with $p$ a fine potential and $h$ invariant.
Then $h\le\widehat R{}_{K(.,Y)}^E\le K(.,Y)$ and hence by \cite[Lemma 2.2]{EF1}
$h\preccurlyeq K(.,Y)$,
which shows that $h=\alpha K(.,Y)$ for some $\alpha\in[0,1]$. Here
$\alpha\ne1$, for otherwise $(h=)$ $\widehat R{}_{K(.,Y)}^E=K(.,Y)$ contrary to
hypothesis. On the other hand it follows by 6.\ (with $A=B$) and 4.\ in
Proposition \ref{prop4.6} that
$$\widehat R{}_{K(.,Y)}^E=\widehat R{}_{\widehat R{}_{K(.,Y)}^E}^E=\widehat R{}_{p+h}^E
=\widehat R{}_p^E+\widehat R{}_h^E=p+h,$$
whence $\widehat R{}_p^E=p$ and $\widehat R{}_h^E=h$. If $h\ne0$ then
$\alpha\ne0$ because $h=\alpha K(.,Y)$. Since
$h=\widehat R{}_h^E=\alpha\widehat R{}_{K(.,Y)}^E=\alpha p+\alpha h$
we would obtain $(1-\alpha)h=\alpha p$ with $0<\alpha<1$, which is impossible.
Thus actually $h=0$, and so indeed $\widehat R{}_{K(.,Y)}^E=p\in{\cal P}(U)$.
\end{proof}

\begin{definition}\label{def5.2}  A set $E\subset U$ is said to be
minimal-thin at a point $Y\in\Delta_1(U)$ if
$\widehat R^E_{K(.,Y)}\ne K(.,Y)$, or equivalently if $R^E_{K(.,Y)}\ne K(.,Y)$,
that is (by the preceding lemma) if $\widehat R{}_{K(.,Y)}^E\in\cal P(U)$.
\end{definition}

\begin{cor}\label{cor5.3} For any $Y\in\Delta_1(U)$ the sets $E\subset U$
which are minimal-thin at $Y$ form a filter $\cal F(Y)$ on $U$.
 \end{cor}

This follows from Lemma \ref{lemma5.1} which easily implies that for any
$E_1,E_2\subset U$ such that ${\widehat R}{}_{K(.,Y)}^{U\setminus E_i}\ne K(.,Y)$
for $i=1,2$, we have ${\widehat R}{}_{K(.,Y)}^{U\setminus(E_1\cup E_2)}\ne K(.,Y)$.

Like in classical potential theory we define the minimal-fine (mf) topology
on $\overline U$ as follows:

\begin{definition}\label{def5.4} A set $W\subset \overline U$ is said
to be a minimal-fine neighborhood of a point $Y\in\overline U$ if

(a) $W\cap U$ is a fine neighborhood of $Y$ in the usual sense, in case
$Y\in U$,

(b) $W$ contains the point $Y$ and $U\setminus W$ is minimal-thin at $Y$, in
case $Y\in\Delta_1(U)$,

(c) $W$ contains the point $Y$, in case $Y\in\Delta(U)\setminus\Delta_1(U)$.
\end{definition}

In the sequel we will denote by $\mflim$ and $\mfliminf$ the limit and
the $\liminf$ in the sense of the mf-topology.

According to (a) above, the mf-topology on $\overline U$ induces on
$U$ the fine topology there, and $U$ is mf-open in $\overline U$, that is,
$\Delta(U)$ is mf-closed in $\overline U$
(since $\varnothing$ is minimal-thin at any point $Y\in\Delta_1(U)$).
According to (c), the mf-topology on $\overline U$ induces the discrete
topology on $\Delta(U)\setminus\Delta_1(U)$.
In view of (c), (b), and Definition \ref{def5.2}, $\Delta_1(U)$ is the set of
mf-limit points of $U$ in $\overline U$.

\begin{prop}\label{prop5.5} The {\rm mf}-topology on $\overline U$ is finer
than the natural topology (and is therefore Hausdorff).
\end{prop}

\begin{proof} Let $W$ be a natural neighborhood of a point
$Y\in\overline U$.
If $Y\in U$ then $W\cap U$ is a usual fine neighborhood of $Y$ in $U$
according to Lemma \ref{lemma4.2} (c), and hence an mf-neighborhood of $Y$
in $\overline U$ by Definition \ref{def5.4} (a) above.
If $Y\in\Delta(U)\setminus\Delta_1(U)$
there is nothing to prove in view of (c) in that definition.
In the remaining case where $Y\in\Delta_1(U)$
we show that $U\setminus W$ is minimal-thin at $Y$, cf.\ Definition
\ref{def5.4} (b), which by Lemma \ref{lemma5.1} means that
${\widehat R}{}_{K(.,Y)}^{U\setminus W} \ne K(.,Y)$.
Suppose that, on the contrary, ${\widehat R}{}_{K(.,Y)}^{U\setminus W}=K(.,Y)$.
Writing $U\setminus W=A$ we have $Y\in W\subset\complement\overline A$
(complement relative to $\overline U$) because $W$ is open.
It follows by Proposition \ref{prop4.16} (iii) that
$\widehat R{}_{K(.,Y)}^A\ne K(.,Y)$, which is contradictory.
\end{proof}

\begin{definition}\label{def4.14}
Let $h$ be a non-zero minimal invariant function. A point
$Y\in\overline U$ is termed a pole of $h$ if ${\widehat R}{}_h^{\{Y\}}=h.$
\end{definition}

\begin{remark}\label{remark4.15} Any pole $Y$ of $h$ belongs to $\Delta(U)$,
for if \,$Y\in U$ then ${\widehat R}{}_h^{\{Y\}}=0$ because $\{Y\}$ is polar.
\end{remark}

\begin{theorem}\label{thm4.17}
Every non-zero minimal invariant function on $U$ has precisely one pole.
For any $Y\in\Delta_1(U)$ the pole of $K(.,Y)$ is $Y$.
\end{theorem}

\begin{proof} Recall from \cite[Proposition 3.6]{EF1} (and the beginning of
\cite[Section 3]{EF1}) that the non-zero minimal invariant functions on $U$ are
precisely the functions of the form $K(.,Y)$ for a (unique) $Y\in\Delta_1(U)$.
Consider the family ${\cal C}$ of all (necessarily nonvoid) compact subsets
$C$ of ${\overline A}$ such that
${\widehat R}{}_{K(.,Y)}^C=K(.,Y)$, and note that $\cal C$ is nonvoid, for
$\overline U\in\cal C$ because ${\widehat R}{}_{K(.,Y)}^{\overline U}=K(.,Y)$
according to Proposition \ref{prop4.16} (ii), (iii).
Equip $\cal C$ with the order defined by the inverse inclusion `$\supset$'.
For any totally ordered subfamily $\cal C'$ of $\cal C$ the
intersection $C'$ of $\cal C'$ satisfies ${\widehat R}{}_{K(.,Y)}^{C'}=K(.,Y)$
in view of the fundamental convergence theorem,
and hence $\cal C$ has a minimal element $C_0$ according to Zorn's lemma.
The natural topology is Hausdorff, so if $C_0$ contains two
distinct points $Z_1$ and $Z_2$ then there are compact subsets
$C_1$ and $C_2$ of $C_0$ such that
$C_0=C_1\cup C_2$, $Z_1\in C_0\setminus C_2$ and $Z_2\in C_0\setminus C_1$.
Since  $K(.,Y)$ is extreme it then follows by Riesz decomposition
that either ${\widehat R}{}_{K(.,Y)}^{C_1}=K(.,Y)$ or
${\widehat R}{}_{K(.,Y)}^{C_2}=K(.,Y)$.
In other words, either $C_1$
or $C_2$ belongs to $\cal C$, say $C_1\in\cal C$. By minimality of
$C_0$ we would then have $C_1=C_0$ which contradicts
$Z_2\in C_0\setminus C_1$. This shows that indeed $C_0=\{Z\}\in\cal C$
for a certain $Z\in\overline U$, that is ${\widehat R}{}_{K(.,Y)}^{\{Z\}}=K(.,Y)$.
Thus $Z$ is a pole of $K(.,Y)$. Since $\{Z\}$ is a closed set
it follows by Proposition \ref{prop4.14} and Choquet's theorem that
${\widehat R}{}_{K(.,Y)}^{\{Z\}}=K\mu$ for some
probability measure $\mu$ on the compact base $B$ of the cone $\cal S(U)$
such that $\mu$ is carried by $\{Z\}$, that is, for $\mu=\eps_Z$.
Thus $K(.,Y)=K\mu=K(.,Z)$, and so indeed $Y=Z$.
\end{proof}

\begin{prop}\label{prop4.18}
Let $h$ be a non-zero minimal invariant function on $U$ with pole
$Y\in \Delta_1(U)$ and let $A\subset \Delta(U)$. Then
${\widehat R}{}_h^A= h$ or $0$ if \,$Y\in A$ or $Y\notin A$, respectively.
\end{prop}

\begin{proof}
By Proposition \ref{prop4.7} the function ${\widehat R}{}_h^A$ is invariant
and $\preccurlyeq K(.,Y)$, hence of the form $cK(.,Y)$ for some
constant $c$. It follows by 6.\ in Proposition \ref{prop4.6} with $A=B$
that $c=0$ or $c=1$. If $A\subset \Delta(U)$ contains the pole $Y$
of $h$ then $h\ge {\widehat R}{}_h^A\ge {\widehat R}{}_h^{\{Y\}}=h$,
and so ${\widehat R}{}_h^A=h$.
The rest follows from Proposition \ref{prop4.16} (ii).
\end{proof}

The following integral representation of the sweeping of a function of
class $\cal S(U)$ on arbitrary sets $A\subset\overline U$ is based on
Proposition \ref{prop4.18}, which in turn depended on Proposition
\ref{prop4.7}.

\begin{theorem}\label{thm5} For any set $A\subset\overline U$
and any Radon measure $\mu$ on $\overline U$ carried by $U\cup\Delta_1(U)$
we have
\begin{eqnarray}\widehat R{}_{K\mu}^A\!\!\!
&=&\!\!\! \int^*\widehat R{}_{K(.,Y)}^Ad\mu(Y).
\end{eqnarray}
If $A$ is $\mu$-measurable then
the upper integral becomes a true integral.
\end{theorem}

\begin{proof} For any subset $A$ of $U$ this integral representation
was established in \cite[Lemma 3.21]{EF1}
with the upper integral replaced by the integral. For $A\subset\Delta(U)$
it suffices to consider the case where $\mu$ is carried by $\Delta_1(U)$,
for if $\nu$ denotes the restriction of $\mu$ to $U$ then $K\nu$ and
$K(.,Y)$ (for $Y\in U$) are fine potentials according to
\cite[Corollary 3.25]{EF1}, and so
$\widehat R{}_{K\nu}^A=\widehat R{}_{K(.,Y)}^A=0$
by Proposition \ref{prop4.16} (i).

{\it{Proof that the inequality}} `$\ge$' {\it{holds in}}(3.1){\it{ for }}
$A\subset\Delta(U)$. We may assume that
$\widehat R{}_{K\mu}^A\not\equiv+\infty$, that is
$\widehat R{}_{K\mu}^A\in\cal S(U)$.
By Remark \ref{remark2.4} there is a decreasing
sequence $(W_j)$ of sets of class $\cal W(A)$ such that it suffices in
Definition \ref{def4.5} to take for $W\in\cal W(A)$ the sets $W_j$.
We show that the following equations and inequality hold quasieverywhere
on $U$:
\begin{eqnarray*}{\widehat R}{}_{K\mu}^A\!\!\!
&\underset1=&\!\!\!\underset{j}{\widehat\inf}\,{\widehat R}{}_{K\mu}^{W_j\cap U}
\underset2=\inf_j{\widehat R}{}_{K\mu}^{W_j\cap U}
\underset3=\inf_j\int{\widehat R}{}_{K(.,Y)}^{W_j\cap U}d\mu(Y)\\
&\underset4=&\!\!\!\int\inf_j{\widehat R}{}_{K(.,Y)}^{W_j\cap U}d\mu(Y)
\underset5=
\int\underset{j}{\widehat\inf}\,{\widehat R}{}_{K(.,Y)}^{W_j\cap U}d\mu(Y)
\underset6\ge\int^*{\widehat R}{}_{K(.,Y)}^Ad\mu(Y).
\end{eqnarray*}
When these relations have been established quasieverywhere on $U$, the desired
resulting inequality holds everywhere on $U$. In fact,
${\widehat R}{}_{K\mu}^A\in\cal S(U)$ along with $K\mu$; and by Proposition
\ref{prop4.18} we have since $\mu$ is carried by $\Delta_1(U)$
\begin{eqnarray*}\int^*{\widehat R}{}_{K(.,Y)}^Ad\mu(Y)\!\!\!
&=&\!\!\!\int^*K(.,Y)1_A(Y)d\mu(Y)\\
&=&\!\!\!\int K(.,Y)1_{A^*}(Y)d\mu(Y)
=K(1_{A^*}\mu)\in\cal S(U),
\end{eqnarray*}
where $A^*\subset\overline U$ denotes a $G_\delta$ set containing $A$
such that $\mu^*(A^*\setminus A)=0$, cf.\ \cite[Theorem 3.20]{EF1}.
Equation 1 and inequality 6 hold everywhere on $U$ by Definition \ref{def4.5}.
Eq.\ 2 holds quasieverywhere\ by the fundamental convergence
theorem \cite[Theorem 11.8]{F1}.
Eq.\ 3 holds at any point $x\in U$ at which
$K\mu(x)<+\infty$ and hence $\widehat R{}_{K\mu}^A(x)<+\infty$,
for there we have by \cite[Lemma 3.21]{EF1}
$\widehat R{}_{K\mu}^{W_j\cap U}(x)=\int\widehat R{}_{K(.,Y)}^{W_j\cap U}(x)d\mu(Y)$,
which is finite for large $j$ (depending on $x$).
Eq.\ 4 is obvious (Lebesgue) at points $x$ as stated for eq.\ 3.
In the first place,
${\widehat R}{}_{K(.,Y)}^{W_j\cap U}$ is of class $\cal S(U)$
for each $Y\in\Delta_1(U)$, hence
finely continuous and in particular Borel measurable on $U$ according to
\cite[Lemma 2.1]{EF1}. Secondly, the integrals are finite, being majorized by
$\int K(.,Y)d\mu(Y)=K\mu<+\infty$ at points $x$ as stated.
Concerning the remaining  eq.\ 5, note that for each $k$ the function
${\widehat R}{}_{K(.,Y)}^{W_k\cap U}$ is invariant on
$U\setminus{\widetilde W}{}_k$ according to
\cite[Lemma 2.4]{EF1}. For any $j$ and any $k\ge j$
we have $W_k\subset W_j$, and ${\widehat R}{}_{K(.,Y)}^{W_k\cap U}$ is therefore
invariant on each $U\setminus\widetilde{W}{}_j$,
and hence on their union according to
\cite[Theorem 2.6 (a), (b)]{EF1}.
It follows by \cite[Theorem 2.6 (c)]{EF1} that
${\widehat R}{}_{K(.,Y)}^{W_j\cap U}$ is invariant on the finely open set $U$.
For any point $x\in U$ such that $K\mu(x)<+\infty$ the set
$$E_x:\{Y\in\overline U:{\widehat R}{}_{K(.,Y)}^A(x)=+\infty\}$$ is $\mu$-null.
According to\cite[Theorem 2.6 (c)]{EF1} we obtain
$$\underset{j}{\widehat\inf}\,{\widehat R}{}_{K(.,Y)}^{W_j\cap U}(x)
=\inf_j{\widehat R}{}_{K(.,Y)}^{W_j\cap U}(x)
\quad\textrm{$\mu$-a.e.\ for }Y\in\overline U,$$
which implies eq.\ 5  at points $x\in U$ with $K\mu(x)<\infty$.
We have thus shown that
$\widehat R_{K\mu}^A\ge\int^*\widehat R{}_{K(.,Y)}d\mu(Y)$ also for
$A\subset\Delta(U)$.

{\it{The asserted equality in case $A\subset\Delta(U)$}}.
Recall from  Lemma \ref{lemma4.2} (c) the countable cover $(V_k)$  of $U$ by
finely open sets with natural closures $\overline V_k$ in $\overline U$
contained in $U$. The complements $D_k:=\complement\overline V_k$
(relative to $\overline U$) form a decreasing sequence of open subsets of
$\overline U$ with the intersection $\Delta(U)$, and such that the
mf-closures $\widetilde D_k$ likewise have the intersection $\Delta(U)$.
Consider first the case where $A=C\cap\Delta(U)$, $C$ compact in
$\overline U$.
Choose a decreasing sequence of open subsets
$C_k$ of $\overline U$ containing $C$ such that $\bigcap_k\overline C{}_k=C$.
Suppose to begin with that $\mu$ is carried by some compact
set $E\subset\Delta(U)\setminus A=\Delta(U)\setminus C$. We may assume
that $E\cap C_k=\varnothing$. The decreasing open sets
$W_k:=C_k\cap D_k\supset A$ are of class $\cal W(A)$ and hence
\begin{eqnarray}\widehat R{}_{K\mu}^A
\le\widehat R{}_{K\mu}^{W_k\cap U}
=\int_E\widehat R{}_{K(.,Y)}^{W_k\cap U}d\mu(Y).
\end{eqnarray}
For each $Y\in E$ the functions $p_k:=\widehat R{}_{K(.,Y)}^{W_k\cap U}$ are
fine potentials on $U$ according to Lemma \ref{lemma5.1} because the
mf-closure $\widetilde W{}_k$ of $W_k$ is contained in
$\widetilde C{}_k\cap\widetilde D_k\subset\overline C{}_k\cap\widetilde D{}_k$
according to Proposition \ref{prop5.5}, and hence does not meet $E$ and
in particular does not contain $Y$. It follows that
$p:=\widehat\inf{}_kp_k$ is a fine potential on $U$.
By \cite[Lemma 2.4]{EF1} the restriction of $p_k$ to the finely open set
$U\setminus\widetilde W{}_k$ is invariant. By \cite[Theorem 2.5]{EF1} (b)
it follows that so is the restriction of $p_k$ to
$$
{\bigcup}_{l\ge k}U\setminus\widetilde W{}_l
=U\setminus{\bigcap}_{l\ge k}\widetilde W{}_l
\supset U\setminus{\bigcap}_{l\ge k}\overline C{}_l\cap\widetilde D{}_l
=U\setminus A=U.
$$
By \cite[Theorem 2.6]{EF1} (c) we infer that $p$ is itself invariant on
$U$, and being also a fine potential $p$ must be $0$. It therefore follows
by (3.2) that $\widehat R{}_{K\mu}^A
\le\int_E\widehat R{}_{K(.,Y)}^{W_k\cap U}d\mu(Y)\searrow0$ pointwise on $U$,
and hence $\widehat R{}_{K\mu}^A=0\le\int\widehat R{}_{K(.,Y)}^Ad\mu(Y)$.
In combination with the opposite inequality in (3.1) obtained above (now
with a true integral) this establishes equality in (3.1) in the present
case where $A=C\cap\Delta(U)$ with $C$ compact in $\overline U$ and $\mu$
carried by a compact set $E\subset\Delta(U)\setminus A$.

Next, replace the latter assumption on $\mu$ by the weaker temporary
assumption that $\mu(A)=0$. Choose an increasing sequence of compact sets
$E_j\subset\Delta(U)\setminus A$ such that $\mu(E_j)\nearrow\mu(A)$,
and denote by $\mu_j$ the part of $\mu$ on $E_j$. By Proposition
\ref{prop4.17} it follows that
$$\widehat R{}_{K\mu}^A=\sup_j\widehat R{}_{K\mu_j}^A
=\sup_j\int\widehat R{}_{K(.,Y)}^Ad\mu_j=\int\widehat R{}_{K(.,Y)}^Ad\mu(Y)=0
$$
according to Proposition \ref{prop4.16} (ii).

Without any such temporary assumption on $\mu$ we denote by $\mu_A$ and
$\mu'$ the parts of $\mu$ on $A$ and on $\overline U\setminus A$,
respectively. Then $\mu'(A)=0$ and hence
$$\widehat R{}_{K\mu}^A=\widehat R{}_{K\mu_A}^A+\widehat R{}_{K\mu'}^A
\le K\mu_A+\int\widehat R{}_{K(.,Y)}^Ad\mu'(Y)
=\int\widehat R_{K(.,Y)}^Ad\mu(Y).
$$
When combined with the inequality `$\ge$' in (3.1) obtained above
(with an upper integral) this leads to equality in (3.1) (with a true
integral) for arbitrary $\mu$ when $A=C\cap\Delta(U)$ with $C$ compact.

More generally, if $A=C\cap\Delta(U)$  and if $C$ is just the union of an
increasing sequence of compact sets $C_j\subset\overline U$, then
$$
\widehat R{}_{K\mu}^{C_j\cap\Delta(U)}
=\int\widehat R{}_{K(.,Y)}^{C_j\cap\Delta(U)}d\mu(Y)
\le\int\widehat R{}_{K(.,Y)}^Ad\mu(Y).
$$
For $j\to\infty$ it follows by (a) in Proposition \ref{prop4.7a} that
$\widehat R{}_{K\mu}^A\le\int\widehat R{}_{K(.,Y)}^Ad\mu(Y)$.
Together with the opposite inequality obtained above this leads to (3.1)
(with a true integral) for any set $A=C\cap\Delta(U)$ with $C$ a $K_\sigma$
subset of $\overline U$.
This applies in particular to $A=C\cap\Delta(U)$ with $C$ open in
$\overline U$.

Next, let $A$ be any $G_\delta$ subset of $\Delta(U)$. Since $\Delta(U)$
is itself a $G_\delta$ in $\overline U$ this means that $A=C\cap\Delta(U)$
for some $G_\delta$ subset $C$ of $\overline U$. Thus $C$ is the
intersection of a decreasing sequence of open sets
$C_j\subset\overline U$. Denote by $\mu_j$ the part of $\mu$ on
$\Delta(U)\setminus A_j$. Then $\mu_j(A_j)=0$, and again, since $A_j$ is
a $K_\sigma$ subset of $\Delta(U)$,
$$\widehat R{}_{K\mu_j}^A\le\widehat R{}_{K\mu_j}^{A_j}
=\int\widehat R{}_{K(.,Y)}^{A_j}d\mu_j(Y)=0
$$
according to Proposition \ref{prop4.16} (ii).
But $\widehat R{}_{K\mu_j}^A\nearrow\widehat R{}_{K\mu}^A=0$ in the specific
order by Proposition \ref{prop4.17}, whence the assertion.

Finally, let $A$ be an arbitrary subset of $\Delta(U)$. Then $A$ can be
extended by a $\mu$-nullset to a $G_\delta$
set $A^*\subset\Delta(U)$ because $\Delta(U)$ is itself a $G_\delta$.
We obtain the missing inequality `$\le$' as follows:
\begin{eqnarray*}\widehat R{}_{K\mu}^A\!\!\!
&\le&\!\!\!\widehat R{}_{K\mu}^{A^*}
=\int\widehat R{}_{K(.,Y)}^{A^*}d\mu(Y)
=\int K(.,Y)1_{A^*}(Y)d\mu(Y)\\
&=&\!\!\!\int^*K(.,Y)1_A(Y)d\mu(Y)
=\int^*\widehat R{}_{K(.,Y)}^Ad\mu(Y)
\end{eqnarray*}
according to Proposition \ref{prop4.18}. We have thus shown that (3.1) holds
for any set $A\subset\Delta(U)$. According to the last equality the upper
integral in the above display becomes a true integral if the subset $A$
of $\Delta(U)$ is $\mu$-measurable.


{\it{The general case of the theorem}}. By Propositions \ref{prop4.8} and
\ref{prop4.18} we have
$$
\widehat R{}_{K\mu}^A
=\widehat R{}_{K\mu}^{A\cap\Delta(U)}+\widehat R{}_v^{A\cap U}
=\int^*\widehat R{}_{K(.,Y)}^{A\cap\Delta(U)}d\mu(Y)+\widehat R{}_v^{A\cap U},
$$
where
\begin{eqnarray*}v:\!\!\!
&=&\!\!\!K\mu-\widehat R{}_{K\mu}^{A\cap\Delta(U)}
\;=\;
\int K(.,Y)d\mu(Y)-\int^*\widehat R{}_{K(.,Y)}^{A\cap\Delta(U)}d\mu(Y)\\
&=&\!\!\!\int_*K(.,Y)(1-1_{A^*\cap\Delta_1(U)}(Y))d\mu(Y)
\;=\;K\lambda,
\end{eqnarray*}
with $\lambda:=(1-1_{A^*\cap\Delta_1(U)})\mu$, $A^*$ denoting again a $G_\delta$
set containing $A$ such that $\mu^*(A^*\setminus A)=0$.
Since $\lambda\le\mu$, $\lambda$ is carried by $\Delta_1(U)$.
It follows that
$$
\widehat R{}_v^{A\cap U}
=\int\widehat R{}_{K(.,Y)}^{A\cap U}d\lambda(Y)
=\int\widehat R{}_{K(.,Y)}^{A\cap U}(1-1_{A^*\cap\Delta_1(U)}(Y))d\mu(Y),
$$
\begin{eqnarray}\widehat R{}_{K\mu}^A\!\!\!
&=&\!\!\!\int\bigl(K(.,Y)1_{A^*\cap\Delta_1(U)}(Y)
+\widehat R{}_{K(.,Y)}^{A\cap U}(1\!-\!1_{A^*\cap\Delta_1(U)}(Y))\bigr)d\mu(Y).
\end{eqnarray}

\vskip2truemm\noindent
Similarly, for any $Y\in \Delta_1(U)$,
$$\widehat R{}_{K(.,Y)}^A
=\widehat R{}_{K(.,Y)}^{A\cap\Delta(U)}+\widehat R{}_{v_Y}^{A\cap U}
=1_{A\cap\Delta_1(U)}(Y)K(.,Y)+\widehat R{}_{v_Y}^{A\cap U},
$$
$$v_Y:
=K(.,Y)-\widehat R{}_{K(.,Y)}^{A\cap\Delta(U)}
=(1-1_{A\cap\Delta_1(U)}(Y))K(.,Y)\;=\;K\lambda_Y
$$
with $\lambda_Y:=(1-1_{A^*\cap\Delta_1(U)}(Y))\mu$ carried by $\Delta_1(U)$.
It follows that
\begin{eqnarray*}\widehat R{}_{v_Y}^{A\cap U}
&=&\int\widehat R{}_{K(.,Z)}^{A\cap U}d\lambda_Y(Z)
=(1-1_{A\cap\Delta_1(U)}(Y))\widehat R{}_{K\mu}^{A\cap U},
\end{eqnarray*}
\begin{eqnarray*}\widehat R{}_{K(.,Y)}^A
&=&1_{A\cap\Delta_1(U)}(Y)K(.,Y)+(1-1_{A\cap\Delta_1(U)}(Y))\widehat R{}_{K(.,Y)}^{A\cap U}.
\end{eqnarray*}
The upper integral of this expression for $\widehat R{}_{K(.,Y)}^A$
with respect to $d\mu(Y)$ is just the right hand member of (3.3)
because $1_{A\cap\Delta_1(U)}=1_{A^*\cap\Delta_1(U)}$ $\mu$-a.e. This proves that
indeed $\widehat R{}_{K\mu}^A=\int^*\widehat R{}_{K(.,Y)}^Ad\mu(Y)$.
\end{proof}

\begin{cor}\label{cor5} Let $u\in\cal S(U)$, let $A\subset\Delta(U)$, and let
$\mu$ denote the (unique) representing measure for the invariant function
$\widehat R{}_u^A$ carried by $\Delta_1(U)$. Then
$\mu$ is carried by $A$ (in the sense that $\mu^*(\complement A)=0$)
if and only if $\widehat R{}_u^A=u$.
\end{cor}

\begin{proof} By the above theorem together with
Proposition \ref{prop4.18} we have
$$
\widehat R{}_u^A=\int^*\widehat R{}_{K(.,Y)}^Ad\mu(Y)
=\int1_{A^*}(Y)K(.,Y)d\mu(Y)=K(1_{A^*}\mu),
$$
and by uniqueness this equals $u=K\mu$ if and only if $1_{A^*}\mu=\mu$,
which means that $\mu$ shall be carried by $A$.
\end{proof}

The above corollary, sharpening Proposition \ref{prop4.14} (in the present
setting), is an analogue of \cite[Proposition 3.13]{EF1}, where $\mu$ is
carried only by the closure $\overline A$ of $A$, $\overline A$ supposed to
be contained in $U$.

%

In the following minimal-fine boundary minimum property the requirement
$Y\in\Delta_1(U)$ (rather than $Y\in\Delta(U)$) is motivated by the fact
that $\Delta_1(U)$ is the set of all mf-limit points of $U$ in $\overline U$,
as noted after Definition \ref{def5.4}, and so the stated $\mfliminf_{x\to Y}$
is defined for $Y\in\Delta_1(U)$ only.

\begin{prop}\label{prop5.6} Let $u$ be finely super\-harmonic on $U$,
and suppose that
$$\underset{x\to Y,\,x\in U}{\mfliminf}\,u(x)\ge0
\quad\textrm{for every \;}Y\in\Delta_1(U).$$
If moreover $u\ge-s$ on $U$ for some $s\in\cal S(U)$ then $u\ge0$ on $U$.
\end{prop}

\begin{proof} For given $\eps>0$ and $Y\in\Delta(U)$ there exists by the
assumed boundary inequality an mf-fine open mf-neighborhood
$W_Y\subset\overline U$ of $Y$ such that $u>-\eps$ on $W_Y\cap U$.
(Take for example $W_Y=\{Y\}$ for $Y\in\Delta(U)\setminus\Delta_1(U)$,
cf.\ Definition \ref{def5.4} (c).) In terms of the mf-open set
$W:=\bigcup\{W_Y:Y\in\Delta(U)\}$ containing $\Delta(U)$
we infer that $u>-\eps$ on $W\cap U=\bigcup\{W_Y\cap U:Y\in\Delta(U)\}$.
The set $E:=\overline U\setminus W$ is mf-closed and contained
in $U$, and so $E$ is finely closed, as noted before Proposition
\ref{prop5.5}. Furthermore, $E$ is minimal-thin at $Y$ in view of
Definition \ref{def5.4} (b), and hence ${\widehat R}{}_{K(.,Y)}^E$
is a fine potential on $U$ for each $Y\in\Delta(U)$. It follows that
${\widehat R}{}_s^E$ likewise is a fine potential. To see this, write
$s=K\sigma$ and ${\widehat R}{}_{K(.,Y)}^E=K\lambda_Y$ with unique representing
measures $\sigma$ on $U\cup\Delta(U)$ and $\lambda_Y$ on $U$, respectively,
cf.\ \cite[Corollary 3.25]{EF1}. By Lemma \cite[Lemma 3.21]{EF1} we have
\begin{eqnarray*}{\widehat R}{}_s^E\!\!\!
&=&\!\!\!\int_{U\cup\Delta_1(U)}{\widehat R}{}_{K(.,Y)}^Ed\sigma(Y)
=\int_{U\cup\Delta_1(U)}K\lambda_Yd\sigma(Y)\\
&=&\!\!\! \int_{U\cup\Delta_1(U)}\Bigl(\int_UK(.,z)d\lambda_Y(z)\Bigr)d\sigma(Y)
=\int_UK(.,z)d\nu(z),
\end{eqnarray*}
where the measure $\nu=\int_{U\cup\Delta_1(U)}\lambda_Yd\sigma(Y)$ on $U$ is the
integral with respect to $\sigma$ of the family
$(\lambda_Y)_{Y\in U\cup\Delta_1(U)}$ of measures on $U$,
cf.\ \cite[3, proposition 1]{Bo}.
In particular, $\nu$ is carried by $U$ along with each $\lambda_Y$. Hence
${\widehat R}{}_s^E$ is indeed a fine potential according to
\cite[Corollary 3.26]{EF1}. The (possibly empty) fine interior $V$ of $E$
has relative fine boundary
$U\cap\partial_fV\subset U\setminus V\subset U\cap\widetilde W$.
We have $u+\eps\ge0$ on $U\cap W$ and hence by fine continuity on
$U\cap\widetilde W\supset U\cap\partial_fV$.
Furthermore, $u+{\widehat R}{}_s^E=u+s$ on $U\cap b(E)\supset V$,
where $b(E)$ denotes the base of $E$ in $\Omega$; and so
$u+\eps\ge-{\widehat R}{}_s^E$ on $V$. Altogether, it follows by the
relative fine boundary minimum property \cite[Theorem 10.8]{F1}
applied to the fine potential $p={\widehat R}{}_s^E$ that $u+\eps\ge0$
on $V$. As noted above, the same inequality holds on
$U\cap\widetilde W\supset U\setminus V$ and thus on all of
$U$. By varying $\eps$ we conclude that indeed $u\ge0$ on all of $U$.
\end{proof}

In view of Proposition \ref{prop5.5} we have the following weaker Martin
boundary minimum property relative to the natural topology on $\overline U$.
Both properties are used in \cite{EF3}.

\begin{cor}\label{cor5.7}  Let $u$ be finely superharmonic on $U$,
and suppose that
$$\underset{x\to Y,\,x\in U}{\lim\inf}\,u(x)\ge0
\quad\textrm{for any }\;Y\in\Delta(U).$$
If moreover $u$ is lower bounded then $u\ge0$ on $U$.
\end{cor}

We proceed to define sweeping on subsets of $\overline U$ relative to the
minimal-fine topology, and to show that
sweeping on $A$ relative to the mf-topology coincides with sweeping on $A$
relative to the natural topology as defined in Definition \ref{def4.5}.

\begin{definition}\label{def4.5a} Let $A\subset\overline U$. For any
function $u\in {\cal S}(U)$ the reduction of $u$ on $A$ relative to the
mf-topology is defined by
$${}^1R{}_u^A
=\inf\{v\in\cal S(U): v\ge u\textrm{ on }A\cap U\textrm{ and on }
W\cap U\textrm{ for some }W\in{}{}^1\cal W(A)\},$$
where ${}^1\cal W(A)$ denotes the family of all mf-open sets
$W\subset\overline U$ such that
$W\supset A\cap \Delta(U)$. The sweeping of $u$ on $A$ is defined as the
greatest finely l.s.c.\ minorant of ${}^1R{}_u^A$ and is denoted by
${}^1{\widehat R}{}_u^A$.
\end{definition}

The function ${}^1{\widehat R}{}_u^A$ is of class $\cal S(U)$.
Similarly to reduction and sweeping relative to the natural topology we have
$${}^1R{}_u^A=\inf\{R{}_u^{(A\cup W)\cap U}:W\in{}^1\cal W(A)\},$$
$${}^1{\widehat R}{}_u^A
={\widehat\inf}\{{\widehat R}{}_u^{(A\cup W)\cap U}:W\in{}^1\cal W(A)\}.$$
Furthermore, there is a decreasing sequence $(W_j)$ of sets
$W_j\in{}^1\cal W(A)$
(depending on $u$) such that it suffices to take for $W$ the sets $W_j$,
in the above definitions and alternative expressions (this is shown in the
same way as in the case of sweeping relative to the natural topology
by application of the fundamental convergence theorem  and the
quasi-Lindel{\"o}f property for finely u.s.c.\ functions).
For any subset $A$ of $U$, the present reduction ${}^1R{}_u^A$ and
sweeping ${}^1{\widehat R}{}_u^A$ on $A$ relative to $\overline U$ clearly
reduce to the usual reduction and sweeping on $A$ relative to $U$.
Since the mf-topology is finer than the natural topology
(Proposition \ref{prop5.5}), we clearly have
${}^1R{}_u^A\le R{}_u^A$ and ${}^1{\widehat R}{}_u^A\le {\widehat R}{}_u^A$.

As in the Euclidean case \cite[Theorem 8.3.1]{AG} we have the following

\begin{cor}\label{cor6}
For any $u\in\cal S(U)$ we have
${}^1{\widehat R}{}_u^{\Delta(U)\setminus\Delta_1(U)}=0$.
\end{cor}

\begin{proof} According to Definition \ref{def5.4} (c) the set
$A:=\Delta(U)\setminus\Delta_1(U)$ belongs to ${}^1\cal W(A)$, and
$A\cap U=\varnothing$, whence ${}^1{\widehat R}{}_u^A=0$.
\end{proof}

We shall need the following analogue of Proposition \ref{prop4.18}:

\begin{lemma}\label{lemma5.8}
For any $A\subset\Delta(U)$ and $Y\in U\cup\Delta_1(U)$ we have
${}^1{\widehat R}{}_{K(.,Y)}^A=K(.,Y)$ if \,$Y\in A$,
and ${}^1{\widehat R}{}_{K(.,Y)}^A=0$ if  $Y\notin A$.
\end{lemma}

\begin{proof}  If $Y\notin A$ then
 ${}^1{\widehat R}{}_{K(.,Y)}^A\le{\widehat R}{}_{K(.,Y)}^A=0$ by Proposition
 \ref{prop4.16} (ii).
If $Y\in A$ and hence $Y\notin U$, then $Y\in\Delta_1(U)$, and
${}^1{\widehat R}{}_{K(.,Y)}^A=K(.,Y)$
because we even have ${}^1{\widehat R}{}_{K(.,Y)}^{\{Y\}}=K(.,Y)$. In fact,
for any $W\in{}^1\cal W(\{Y\})$,
 $$K(.,Y)=\widehat R{}_{K(.,Y)}^U
\le\widehat R{}_{K(.,Y)}^{U\cap W}+\widehat R{}_{K(.,Y)}^{U\setminus W},$$
where the latter term on the right is a fine potential on $U$ by Definition
\ref{def5.2}, $U\setminus W$ being minimal-thin at $Y$ in view of Definition
\ref{def5.4} (b). By the Riesz decomposition property
we obtain $K(.,Y)=u+v$ with $u\le\widehat R{}_{K(.,Y)}^{U\cap W}$ and
$v\le\widehat R{}_{K(.,Y)}^{U\setminus W}$.
This shows that $v\preccurlyeq K(.,Y)$
and hence $v=0$, $v$ being a fine potential along with
$\widehat R{}_{K(.,Y)}^{U\setminus W}$, and $K(.,Y)$ being invariant since
$Y\in\Delta_1(U)$. Thus $K(.,Y)=u\le\widehat R{}_{K(.,Y)}^{U\cap W}$, obviously
with equality. By varying $W$ we infer by Definition \ref{def4.5a} that
indeed ${}^1\widehat R{}_{K(.,Y)}^{\{Y\}}=K(.,Y)$.
\end{proof}

The six assertions of Proposition \ref{prop4.6} carry over along with their
proofs when reductions and sweepings are taken with respect to the
minimal-fine topology on $U$ instead of the smaller natural topology, and of
course $\cal W(A)$ is replaced by ${}^1\cal W(A)$ for any
$A\subset\overline U$. The same
applies to Propositions \ref{prop4.7}, \ref{prop4.7a}, and \ref{prop4.8}.

\begin{theorem}\label{thm5.16} Let $A\subset\overline U$ and $u\in\cal S(U)$.
Then  ${}^1{\widehat R}{}_u^A={\widehat R}{}_u^A$.
\end{theorem}

\begin{proof} This is obvious if $A\subset U$. Next, for $A\subset\Delta(U)$,
write $u=K\mu$ with $\mu$ carried by $U\cup\Delta_1(U)$.
For any $W\in{}^1\cal W(A)$ we have by \cite[Lemma 3.21]{EF1}
\begin{eqnarray*}\widehat R{}_{K\mu}^{W\cap U}\!\!\!
&=&\!\!\!\int\widehat R{}_{K(.,Y)}^{W\cap U}d\mu(Y)
\ge\int^*{}^1\widehat R{}_{K(.,Y)}^Ad\mu(Y)\\
&=&\!\!\!\int^*\widehat R{}_{K(.,Y)}^Ad\mu(Y)=\widehat R{}_{K\mu}^A,
\end{eqnarray*}
the second equality because
${}^1\widehat R{}_{K(.,Y)}^A=\widehat R{}_{K(.,Y)}^A=1_A(Y)K(.,Y)$ for
$Y\in U\cup\Delta_1(U)$ according to Lemma \ref{lemma5.8} and Proposition
\ref{prop4.18}, respectively; and the third equality follows by
Theorem \ref{thm5}. By varying $W\in{}^1\cal W(A)$ this yields
${}^1{\widehat R}{}_{K\mu}^A\ge{\widehat R}{}_{K\mu}^A$,
actually with equality.
It follows by Proposition \ref{prop4.8} and its mf version that indeed
${}^1{\widehat R}{}_u^A={\widehat R}{}_u^A$ for any $A\subset\overline U$
because $v$ is the same in either
case (by what has just been shown), and hence
${}^1\widehat R{}_v^{A\cap U}=\widehat R{}_v^{A\cap U}$
since $A\cap U\subset U$.
\end{proof}








\end{document}